\documentclass{article}

\usepackage[export]{adjustbox}
\usepackage{amsmath}
\usepackage{amsfonts}
\usepackage{amssymb}
\usepackage{amsthm}
\usepackage{array}
\usepackage{chngpage}
\usepackage{comment}
\usepackage{float}
\usepackage[OT2,T1]{fontenc}
\usepackage{graphicx}
\usepackage{longtable}
\usepackage{mathrsfs}
\usepackage{mathtools}
\usepackage{tikz}
\usepackage[lowtilde]{url}
\usetikzlibrary{matrix,arrows}

\newcommand{\C}{\mathbb{C}}

\newcommand{\N}{\mathbb{N}}

\newcommand{\cO}{\mathcal{O}}

\newcommand{\fa}{\mathfrak{a}}

\DeclareSymbolFont{cyrletters}{OT2}{wncyr}{m}{n}
\DeclareMathSymbol{\sha}{\mathalpha}{cyrletters}{"58}

\newcommand{\eps}{\varepsilon}

\newtheorem{theorem}{Theorem}[section]
\newtheorem{lemma}[theorem]{Lemma}
\newtheorem{corollary}[theorem]{Corollary}
\newtheorem{figurecap}[theorem]{Figure}

\title{Non-Random Behaviour in Sums of Modular Symbols}
\author{Alex Cowan\\cowan@math.columbia.edu}
\date{}

\begin{document}
\maketitle
\begin{abstract}
\noindent We give explicit expressions for the Fourier coefficients of Eisenstein series twisted by Dirichlet characters and modular symbols on $\Gamma_0(N)$ in the case where $N$ is prime and equal to the conductor of the Dirichlet character. We obtain these expressions by computing the spectral decomposition of automorphic functions closely related to these Eisenstein series. As an application, we then evaluate certain sums of modular symbols in a way which parallels past work of Goldfeld, O'Sullivan, Petridis, and Risager. In one case we find less cancellation in this sum than would be predicted by the common phenomenon of ``square root cancellation'', while in another case we find more cancellation.
\end{abstract}
\section{Introduction}\label{introduction}
Let $N$ be a prime, and let $f(z)$ be a cusp form of weight $2$ and level $N$ with Fourier coefficients $a_n$. For $\gamma \in \Gamma_0(N)$, define the modular symbol
$$\langle\gamma,f\rangle := 2\pi i \int_{i\infty}^{\gamma i\infty}f(w)\,dw .$$
Modular symbols have been very useful tools historically and are of significant interest in their own right. Merel in \cite{merel96}, for instance, used modular symbols extensively in his proof that the number of torsion points on an elliptic curve over an arbitrary number field is bounded, and that the bound depends only on the degree of the number field. Another important use of modular symbols is in Cremona's algorithms for elliptic curves \cite{cre97}, which he used to generate parts of the incredibly useful LMFDB \cite{lmfdb}. These algorithms rely on the duality between cusp forms and modular symbols and the resulting action of the Hecke operators on modular symbols, and the information about elliptic curves he derives is obtained by examining the corresponding spaces of modular symbols.\\
\\
There has also been considerable interest in statistical questions regarding modular symbols because they are connected to central values of $L$-functions. Let $\chi$ be a Dirichlet character of conductor $m$, and define $L_f(s,\chi)$ for $\text{Re}(s) > 2$ via the series
$$L_f(s,\chi) := \sum_{n=1}^\infty\frac{\chi(n)a_n}{n^s}.$$
This function has an analytic continuation to all of $\C$. The ``central value'' $L_f(1,\chi)$ contains a large amount of arithmetic information about $f(z)$. The value of $L_f(1,\chi)$ can be given in terms of a finite sum of modular symbols when $\chi$ is primitive:
$$\tau(\bar\chi)L_f(1,\chi) = \sum_{a=1}^m\bar\chi(a)\frac{1}{2}\left(\left\langle\begin{pmatrix}a&*\\m&*\end{pmatrix},f\right\rangle \pm \left\langle\begin{pmatrix}-a&*\\m&*\end{pmatrix},f\right\rangle\right),$$
where $\pm$ is the sign of $\chi$ and $\tau(\bar\chi)$ is the Gauss sum \cite{petris18}. The matrices that appear in this formula are not necessarily in $\Gamma_0(N)$, but nevertheless they are defined via the same expression as before. Based on this formula, Mazur and Rubin \cite{mazrub16} recently made conjectures about the distribution of modular symbols symbols, partially in an attempt to draw a connection to conjectures about ranks of twisted elliptic curves by David, Fearnley, and Kisilevsky \cite{dfk04}. An average version of one of their conjectures was proven by Petridis and Risager in \cite{petris18}, and the full conjecture was proven by Diamantis, Hoffstein, Kiral, and Lee \cite{dhkl18}.\\
\\
Modular symbols are often studied for their own sake as well, especially from a statistical perspective. For $\gamma \in \Gamma_0(N)$, define $|\!|\gamma|\!|_z := |cz+d|^2$, where $c$ and $d$ are the lower-left and lower-right entries of $\gamma$ respectively. Goldfeld \cite{go99} conjectured that
$$\sum_{\gamma:|\!|\gamma|\!|_{i\!M} < X}\left\langle\gamma,f\right\rangle \sim \frac{3}{\pi}\prod_{p|N}\left(1 + \frac{1}{p}\right)^{-1}2\pi i\int_{iM}^{i\infty}f(w)\,dw\,\cdot X.$$
This was proved by Goldfeld and O'Sullivan in \cite{goosul03}.\\
\\
Petridis and Risager then show in \cite{petris04} that modular symbols are normally distributed when ordered by $|\!|\gamma|\!|_z$ for any fixed $z$ in the complex upper half-plane. To prove this result, Petridis and Risager study the properties of an Eisenstein series twisted by modular symbols, which was first defined by Goldfeld in \cite{go99}. This Eisenstein series is defined as
$$E^*(z,s,\chi) := \sum_{\Gamma_\infty\backslash\Gamma_0(N)}\chi(\gamma)\langle\gamma,f\rangle\text{Im}(\gamma z)^s,$$
where $\chi\!\left(\left(\begin{smallmatrix}a&b\\c&d\end{smallmatrix}\right)\right) := \chi(d)$.\\
\\
The Eisenstein series $E^*(z,s,\chi)$ is not automorphic, but does satisfy a certain cocycle relation and can be related to automorphic functions in a simple way. This has led to several papers dedicated entirely to the study of $E^*(z,s,\chi)$. Particularly, O'Sullivan in \cite{osul00} proves that this function has various nice properties such as an analytic continuation and a functional equation, and proves many things about the form of its Fourier expansion. Petridis in \cite{pet02} studies the poles and residues of this Eisenstein series.\\
\\
This paper is focused on the case where $\chi$ has conductor $N$, the level of $f(z)$. In the literature it is usually assumed that the conductor of $\chi$ is coprime to $N$. The techniques used in each case are quite different; surprisingly, the approach presented here fails completely if the conductor of $\chi$ is not exactly $N$. The objects used to obtain the results in this paper can still be defined when the conductor of $\chi$ is different from $N$ but it seems unlikely that they will have any reasonable properties.\\
\\
We prove the following theorem:
\begin{theorem}\label{geometricordering} Let $N$ be a prime, let $\chi$ be an even character of conductor $N$, and let $f$ be a cusp form of weight $2$ for $\Gamma_0(N)$. Let $z = x + iy$ be a complex number with positive imaginary part. If $\chi$ is complex, then
  \begin{align*}
  &\,\,\sum_{\substack{\gamma \in \Gamma_\infty\backslash\Gamma_0(N)\\|\!|\gamma|\!|_z \leq \text{\emph{Im}}(z)X}}\chi(\gamma)\langle\gamma,f\rangle\\
  &=\sum_{\rho: L(\rho,\bar\chi) = 0} \frac{(4\pi)^{1 - \frac{\rho}{2}}}{N}\frac{\Gamma(1 - \rho)}{\Gamma\!\left(1 - \frac{\rho}{2}\right)^2\Gamma\!\left(\frac{\rho}{2}\right)}\frac{L_f(0)L_f(1,\bar\chi)}{L(1 - \rho,\bar\chi)}\underset{\! s = \rho}{\text{\emph{Res }}}\frac{1}{L(s,\bar\chi)}\sum_{n\neq 0}e^{2\pi i n x}\sigma_{\rho}(n,\chi)|n|^{\frac{1-\rho}{2}}y^{\frac{1}{2}}K_{\frac{1}{2}+\rho}(2\pi|n|y) \cdot X^{1 - \frac{\rho}{2}}\\
  &- \sum_{\rho: L(\rho,\chi) = 0} \frac{(4\pi)^{1 - \frac{\rho}{2}}}{N}\frac{\Gamma(1 - \rho)}{\Gamma\!\left(1 - \frac{\rho}{2}\right)^2\Gamma\!\left(\frac{\rho}{2}\right)}\frac{L_f(0)L_f(1,\chi)}{L(1 - \rho,\chi)}\underset{\! s = \rho}{\text{\emph{Res }}}\frac{1}{L(s,\chi)}\sum_{n\neq 0}e^{2\pi i n x}\sigma_{\rho}(n,\bar\chi)|n|^{\frac{1-\rho}{2}}y^{\frac{1}{2}}K_{\frac{1}{2}+\rho}(2\pi|n|y) \cdot X^{1 - \frac{\rho}{2}}\\
    &+ \,\,\,\,\cO(X^{\frac{1}{2}}),
  \end{align*}
  while if $\chi$ is real, then
  $$\sum_{\substack{\gamma \in \Gamma_\infty\backslash\Gamma_0(N)\\|\!|\gamma|\!|_z \leq \text{\emph{Im}}(z)X}}\chi(\gamma)\langle\gamma,f\rangle = \cO(X^{\frac{1}{2}}).$$
\end{theorem}
\noindent The error term in this theorem can be made completely explicit, and we omit it only for brevity. The secondary error term is of size $\cO(X^{\frac{1}{4}})$.
\begin{figure}[H]
\begin{minipage}[t]{0.6\linewidth}
\begin{adjustwidth}{0em}{1em}
This theorem is similar to the theorem of Goldfeld and O'Sullivan \cite{goosul03} mentioned above, but the introduction of the Dirichlet character $\chi$ makes the Eisenstein series $E^*(z,s,\chi)$ have more complicated poles, and this leads to a more complicated result. In particular, it is known that $\langle\gamma,f\rangle \ll |\!|\gamma|\!|_z^\eps$ \cite{petris04}, so it is reasonable to expect that the sum above will be of size $\cO(X^{\frac{1}{2}+\eps})$, based on the commonly observed phenomenon of ``square root cancellation''. However, theorem \ref{geometricordering} suggests that the sum is of size $\cO(X^{\frac{1}{2} + \frac{1}{2}\text{Re}(\rho)})$, where $\rho$ is the rightmost zero of $L(s,\chi)$. The Riemann hypothesis predicts that the rightmost zeros of $L(s,\chi)$ have real part $\frac{1}{2}$, and this would imply that the sum in theorem \ref{geometricordering} is $\cO(X^{\frac{3}{4}})$. It is conceivable that there is additional cancellation in this sum because of some ``coordination'' between the zeros of $L(s,\chi)$, which could cause this sum to be of size $\cO(X^{\frac{1}{2}+\eps})$, but numerical experimentation suggests that this additional cancellation does not occur, and that the sum is indeed of size roughly $X^{\frac{3}{4}}$. Figure \ref{E11afig} illustrates this.\\
\end{adjustwidth}
\end{minipage}
\begin{minipage}[t]{0.4\linewidth}
\centering
\includegraphics[
                 width=0.95\linewidth, height = 5.1cm, valign=t,right]{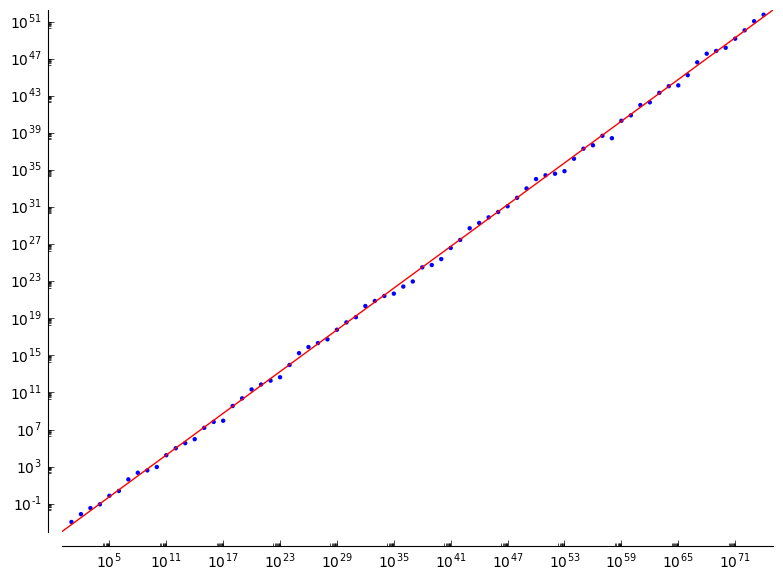}
\begin{adjustwidth}{1em}{0em}
  {\small \begin{figurecap}\label{E11afig} Blue: Absolute value of the main term of theorem \ref{geometricordering} for the cusp form attached to {E11a}, $\chi$ the Dirichlet character modulo $11$ with $\chi(2) = e^{\frac{2\pi i}{5}}$, and $z = i$.\\
      Red: $10^{-4}X^{\frac{3}{4}}$.\end{figurecap}}
\end{adjustwidth}
\end{minipage}%
\end{figure}
\noindent
To prove theorem \ref{geometricordering}, we first compute the Fourier coefficients of $E^*(z,s,\chi)$ in very explicit terms via Selberg spectral decomposition. This is done in section \ref{fourierexpansionssection}, and yields the following two theorems, which may also be of interest in their own right:
\begin{theorem}\label{constterm} Let $N$ be a prime, let $\chi$ be an even character of conductor $N$, and let $f$ be a cusp form of weight $2$ for $\Gamma_0(N)$. Then the constant term of the Fourier expansion for $E^*(z,s,\chi)$ is given by
\begin{align*}
  \int_0^1 E^*(x + iy,s,\chi)\, dx = \Bigg(2\tau(\bar\chi)L_f(1,\chi) \pi^{\frac{1}{2}}N^{-2s}&\frac{\Gamma\left(s-\frac{1}{2}\right)}{\Gamma(s)}\frac{L_f(2s)}{L(2s,\chi)}\\
  &\left.- 2\tau(\chi)L_f(1,\bar\chi) \pi^{\frac{1}{2}}N^{-2s}\frac{\Gamma\left(s-\frac{1}{2}\right)}{\Gamma(s)}\frac{L_f(2s)}{L(2s,\bar\chi)}\right)y^{1-s} .
\end{align*}
\end{theorem}
\begin{theorem}\label{higherterms} Let $N$ be a prime, let $\chi$ be an even character of conductor $N$, and let $f$ be a cusp form of weight $2$ for $\Gamma_0(N)$. Then, for $n\neq 0$, the $n^{\text{th}}$ term of the Fourier expansion for $E^*(z,s,\chi)$ is given by
\begin{align*}
  &\int_0^1E^*(x+iy,s,\chi)e^{-2\pi i n x}dx\\
  =&\sum_j \frac{(4\pi)^{1-s}}{2\langle M_j, M_j\rangle}\frac{\Gamma\left(s+i\lambda_j-\frac{1}{2}\right)\Gamma\left(s-i\lambda_j-\frac{1}{2}\right)}{\Gamma(s)}L\left(s+\frac{1}{2},f\times\overline{M_j}\right) \cdot c_{M_j}(n)y^{\frac{1}{2}}K_{i\lambda_j}(2\pi|n|y)\\
  -&\sum_{m = 1}^{\infty}\frac{a_m}{m}e^{-2\pi m y}\frac{4\pi^s\tau(\chi)}{N^{2s}\Gamma(s)}|n-m|^{\frac{1}{2} - s}\sigma_{2s - 1}(n-m,\bar\chi)y^{\frac{1}{2}}K_{s - \frac{1}{2}}(2\pi|n-m|y)\\
  +&\frac{2^{2s+1}\pi^{2-s}}{N\Gamma(s)}\sum_{k=0}^\infty\frac{\Gamma(2s+k-1)}{k!\Gamma(s+k)\Gamma(1-s-k)}\\
  &\quad\quad\quad\quad\quad\quad\quad\cdot L_f(k)\left(\frac{L_f(1-k,\chi)}{L(2s,\chi)L(2-2s-2k,\chi)}\sigma_{1 - 2s - 2k}(n,\bar\chi) - \frac{L_f(1-k,\bar\chi)}{L(2s,\bar\chi)L(2-2s-2k,\bar\chi)}\sigma_{1 - 2s - 2k}(n,\chi)\right)\\
  &\quad\quad\quad\quad\quad\quad\quad\quad\quad\quad\quad\quad\quad\quad\cdot |n|^{s+k-\frac{1}{2}}y^{\frac{1}{2}}K_{\frac{1}{2}-s-k}(2\pi|n|y)\\
  +&\frac{2^{2s+1}\pi^{2-s}}{N\Gamma(s)}\!\!\sum_{\rho: L(2\rho,\chi) = 0}\! \frac{\Gamma(s+\rho-1)\Gamma(s-\rho)}{\Gamma(1-\rho)\Gamma(\rho)}\frac{L_f(s+1-\rho)L_f(s+\rho,\chi)}{L(2s,\chi)}\underset{\! w = \rho}{\text{\emph{Res }}}\frac{1}{L(2w,\chi)}\,\sigma_{2\rho - 1}(n,\bar\chi)|n|^{\frac{1}{2} - \rho}y^{\frac{1}{2}}K_{\rho - \frac{1}{2}}(2\pi|n|y)\\
  -&\frac{2^{2s+1}\pi^{2-s}}{N\Gamma(s)}\!\!\sum_{\rho: L(2\rho,\bar\chi) = 0}\! \frac{\Gamma(s+\rho-1)\Gamma(s-\rho)}{\Gamma(1-\rho)\Gamma(\rho)}\frac{L_f(s+1-\rho)L_f(s+\rho,\bar\chi)}{L(2s,\bar\chi)}\underset{\! w = \rho}{\text{\emph{Res }}}\frac{1}{L(2w,\bar\chi)}\,\sigma_{2\rho - 1}(n,\chi)|n|^{\frac{1}{2} - \rho}y^{\frac{1}{2}}K_{\rho - \frac{1}{2}}(2\pi|n|y).
\end{align*}
\end{theorem}
~\\
Theorem \ref{constterm} was derived in recent work of Bruggeman and Diamantis \cite{bd16}. They also give an expression for the higher Fourier coefficients of these Eisenstein series, but in terms of a shifted convolution sum involving the Fourier coefficients of $f$ and the sum of divisors function $\sigma_s(n,\chi)$, and this leads to different applications.\\
~\\ 
Theorems \ref{constterm} and \ref{higherterms} make calculations with $E^*(z,s,\chi)$ straightforward. Using standard techniques we then prove theorem \ref{geometricordering} in section \ref{applicationssection}.\\
\\
As another application of theorems \ref{constterm} and \ref{higherterms}, we evaluate certain Kloosterman sums involving modular symbols. The Fourier coefficients of the classical Eisenstein series
$$E(z,s,\chi) := \sum_{\Gamma_\infty\backslash\Gamma_0(N)}\chi(\gamma)\text{Im}(\gamma z)^s$$
are often expressed in terms of the Kloosterman sums
\begin{align*}
  \phi_\chi(n,s) &:= \frac{\pi^s}{\Gamma(s)}|n|^{s-1}\sum_{\left(\begin{smallmatrix}a&b\\c&d\end{smallmatrix}\right) \in \Gamma_\infty\backslash\Gamma_0(N)/\Gamma_\infty}\frac{\chi(d)}{|c|^{2s}}e^{2\pi i n \frac{a}{c}}\\
  \shortintertext{for $n \neq 0$, and}
  \phi_\chi(0,s) &:= \sqrt{\pi}\frac{\Gamma\!\left(s - \frac{1}{2}\right)}{\Gamma(s)}\sum_{\left(\begin{smallmatrix}a&b\\c&d\end{smallmatrix}\right) \in \Gamma_\infty\backslash\Gamma_0(N)/\Gamma_\infty}\frac{\chi(d)}{|c|^{2s}}.\\
\end{align*}
Similarly, the Fourier coefficients of the Eisenstein series twisted by modular symbols $E^*(z,s,\chi)$ can be given in terms of Kloosterman sums twisted by modular symbols \cite{osul00}:
$$E^*(z,s,\chi) = \phi^*_\chi(0,s)y^{1-s} + \sum_{n\neq 0}\phi^*_\chi(n,s)\cdot 2|n|^{\frac{1}{2}}y^{\frac{1}{2}}K_{s-\frac{1}{2}}(2\pi|n|y)e^{2\pi i n x},$$
where
\begin{align*}
  \phi^*_\chi(n,s) &:= \frac{\pi^s}{\Gamma(s)}|n|^{s-1}\sum_{\left(\begin{smallmatrix}a&b\\c&d\end{smallmatrix}\right) \in \Gamma_\infty\backslash\Gamma_0(N)/\Gamma_\infty}\frac{\chi(d)}{|c|^{2s}}e^{2\pi i n \frac{a}{c}}\Big\langle \left(\begin{smallmatrix}a&b\\c&d\end{smallmatrix}\right), f\Big\rangle\\
  \shortintertext{for $n \neq 0$, and}
  \phi^*_\chi(0,s) &:= \sqrt{\pi}\frac{\Gamma\!\left(s - \frac{1}{2}\right)}{\Gamma(s)}\sum_{\left(\begin{smallmatrix}a&b\\c&d\end{smallmatrix}\right) \in \Gamma_\infty\backslash\Gamma_0(N)/\Gamma_\infty}\frac{\chi(d)}{|c|^{2s}}\Big\langle \left(\begin{smallmatrix}a&b\\c&d\end{smallmatrix}\right),f\Big\rangle.
\end{align*}
The Kloosterman sums involved in the classical case can be evaluated, but no evaluation of the Kloosterman sums involving modular symbols above has appeared in the literature. By comparison with theorems \ref{constterm} and \ref{higherterms}, we now obtain closed form expressions for these sums:
\begin{corollary}\label{kloosterman} Let $N$ be a prime, let $\chi$ be an even character of conductor $N$, and let $f$ be a cusp form of weight $2$ for $\Gamma_0(N)$. Then the Kloosterman sums $\phi^*_\chi(n,s)$ appearing above are given by
  $$\phi^*_\chi(0,s) = 2 N^{-2s}L_f(2s)\left(\frac{\tau(\bar\chi)L_f(1,\chi)}{L(2s,\chi)} - \frac{\tau(\chi)L_f(1,\bar\chi)}{L(2s,\bar\chi)}\right)$$
  and, for $n \neq 0$
  \begin{align*}
    &\phi^*_\chi(n,s)\cdot 2|n|^{\frac{1}{2}}y^{\frac{1}{2}}K_{s-\frac{1}{2}}(2\pi|n|y)\\
    &=\sum_j \frac{(4\pi)^{1-s}}{2\langle M_j, M_j\rangle}\frac{\Gamma\left(s+i\lambda_j-\frac{1}{2}\right)\Gamma\left(s-i\lambda_j-\frac{1}{2}\right)}{\Gamma(s)}L\left(s+\frac{1}{2},f\times\overline{M_j}\right) \cdot c_{M_j}(n)y^{\frac{1}{2}}K_{i\lambda_j}(2\pi|n|y)\\
  &-\sum_{m = 1}^{\infty}\frac{a_m}{m}e^{-2\pi m y}\frac{4\pi^s\tau(\chi)}{N^{2s}\Gamma(s)}|n-m|^{\frac{1}{2} - s}\sigma_{2s - 1}(n-m,\bar\chi)y^{\frac{1}{2}}K_{s - \frac{1}{2}}(2\pi|n-m|y)\\
  &+\frac{2^{2s+1}\pi^{2-s}}{N\Gamma(s)}\sum_{k=0}^\infty\frac{\Gamma(2s+k-1)}{k!\Gamma(s+k)\Gamma(1-s-k)}\\
  &\quad\quad\quad\quad\quad\quad\quad\cdot L_f(k)\left(\frac{L_f(1-k,\chi)}{L(2s,\chi)L(2-2s-2k,\chi)}\sigma_{1 - 2s - 2k}(n,\bar\chi) - \frac{L_f(1-k,\bar\chi)}{L(2s,\bar\chi)L(2-2s-2k,\bar\chi)}\sigma_{1 - 2s - 2k}(n,\chi)\right)\\
  &\quad\quad\quad\quad\quad\quad\quad\quad\quad\quad\quad\quad\quad\quad\cdot |n|^{s+k-\frac{1}{2}}y^{\frac{1}{2}}K_{\frac{1}{2}-s-k}(2\pi|n|y)\\
  &+\frac{2^{2s+1}\pi^{2-s}}{N\Gamma(s)}\!\!\sum_{\rho: L(2\rho,\chi) = 0}\! \frac{\Gamma(s+\rho-1)\Gamma(s-\rho)}{\Gamma(1-\rho)\Gamma(\rho)}\frac{L_f(s+1-\rho)L_f(s+\rho,\chi)}{L(2s,\chi)}\underset{\! w = \rho}{\text{\emph{Res }}}\frac{1}{L(2w,\chi)}\,\sigma_{2\rho - 1}(n,\bar\chi)|n|^{\frac{1}{2} - \rho}y^{\frac{1}{2}}K_{\rho - \frac{1}{2}}(2\pi|n|y)\\
  &-\frac{2^{2s+1}\pi^{2-s}}{N\Gamma(s)}\!\!\sum_{\rho: L(2\rho,\bar\chi) = 0}\! \frac{\Gamma(s+\rho-1)\Gamma(s-\rho)}{\Gamma(1-\rho)\Gamma(\rho)}\frac{L_f(s+1-\rho)L_f(s+\rho,\bar\chi)}{L(2s,\bar\chi)}\underset{\! w = \rho}{\text{\emph{Res }}}\frac{1}{L(2w,\bar\chi)}\,\sigma_{2\rho - 1}(n,\chi)|n|^{\frac{1}{2} - \rho}y^{\frac{1}{2}}K_{\rho - \frac{1}{2}}(2\pi|n|y).
  \end{align*}
\end{corollary}
~\\
As a straightforward consequence of corollary \ref{kloosterman}, we can prove the following statement, which parallels theorem \ref{geometricordering}:
\begin{theorem}\label{arithmeticordering} Let $N$ be a prime, let $\chi$ be an even character of conductor $N$, and let $f$ be a cusp form of weight $2$ for $\Gamma_0(N)$. Then
  \begin{align*}
    \sum_{\substack{\gamma \in \Gamma_\infty\backslash\Gamma_0(N)/\Gamma_\infty\\|c| \,<\, X^{\frac{1}{2}}}}\chi(\gamma)\langle\gamma,f\rangle
    \quad=&\quad\sum_{\substack{\rho: L(\rho,\chi) = 0\\\rho\neq 0}} \frac{8\pi i N^{\rho}}{\rho}\tau(\bar\chi)L_f(1,\chi)L_f(\rho)\,\underset{\! s = \frac{\rho}{2}}{\text{\emph{Res }}}\frac{1}{L(2s,\chi)}\cdot X^{\frac{\rho}{2}}\\
    &-\sum_{\substack{\rho: L(\rho,\bar\chi) = 0\\\rho\neq 0}} \frac{8\pi i N^{\rho}}{\rho}\tau(\chi)L_f(1,\bar\chi)L_f(\rho)\,\underset{\! s = \frac{\rho}{2}}{\text{\emph{Res }}}\frac{1}{L(2s,\bar\chi)}\cdot X^{\frac{\rho}{2}}\\
    &+\,4\pi i \left(\tau(\bar\chi)L_f(1,\chi)\frac{L_f'(0)}{L'(0,\chi)} - \tau(\chi)L_f(1,\bar\chi)\frac{L_f'(0)}{L'(0,\bar\chi)}\right),
  \end{align*}
  where $c$ denotes the lower left entry of $\gamma$.
\end{theorem}
\noindent\\
Recent work of Nordentoft \cite{nor18} gives more general results of this type using techniques similar to those used by Petridis and Risager in \cite{petris18}.\\
\\
In the initial work of Goldfeld \cite{go99}, as well as the subsequent work of O'Sullivan \cite{osul00} and Petridis and Risager \cite{petris04}, statistics of modular symbols were considered with the ordering $$\begin{pmatrix}a&b\\c&d\end{pmatrix} \in \Gamma_0(N),\quad c^2 + d^2 < X.$$
More recently, however, Mazur and Rubin \cite{mazrub16} have suggested studying the statistics of modular symbols using the ordering
$$\begin{pmatrix}a&b\\c&d\end{pmatrix} \in \Gamma_0(N),\quad |c| < X,\,\,\, d \,\,(\text{mod}\,c),\,\,\, (c,d) = 1,$$
and much of the recent work on statistics of modular symbols, such as that of Petridis and Risager \cite{petris18}, Diamantis, Hoffstein, Kiral, and Lee \cite{dhkl18}, and Nordentoft \cite{nor18}, uses this ordering instead. We will refer to these orderings as the \textit{geometric} and \textit{arithmetic} orderings of modular symbols respectively.\\
\\
In theorem \ref{geometricordering}, the geometric ordering is used, and, assuming the Riemann hypothesis for Dirichlet $L$-functions and that there is no ``coordination'' between the zeros of $L(s,\chi)$ causing additional cancellation as described above, this geometric ordering leads to a twisted sum of $X$ modular symbols having size roughly $X^{\frac{3}{4}}$. In theorem \ref{arithmeticordering}, however, the arithmetic ordering is used, and the resulting twisted sum of $X$ modular symbols has size roughly $X^{\frac{1}{4}}$. In neither case do we observe ``square root cancellation'', which would yield sums of size $X^{\frac{1}{2}+\eps}$. It is noteworthy, however, that real characters specifically do exhibit square root cancellation under the geometric ordering of theorem \ref{geometricordering}.\\
\\
The phenomenon of square root cancellation arises when random variables are summed in a broad (but not completely universal) sense. Sums in number theory are not over random variables, but square root cancellation is nevertheless observed frequently, because in many situations a sum is over a collection of arithmetic objects ordered in a certain way, and the quantity being summed is unrelated to that ordering. For example, letting $\mu$ denote the M\"{o}bius function, the Riemann hypothesis is equivalent to the statement
$$\sum_{n < X}\mu(n) = \cO(X^{\frac{1}{2}+\eps})$$
for all $\eps > 0$. If one were to sum $X$ random variables which are $1$ half the time and $-1$ half the time, then one would obtain the same result, so the statement above can be viewed as $\mu$ behaving randomly when its integer arguments are ordered by absolute value, or that there is no correlation between the absolute value of an integer and the parity of the number of factors it has. From this perspective, theorems \ref{geometricordering} and \ref{arithmeticordering} are statements that twisted sums of modular symbols behave non-randomly when the modular symbols are ordered using either the geometric ordering or the arithmetic ordering.\\
\subsection*{Acknowledgments} This work was done as part of my PhD thesis under the guidance of Dorian Goldfeld. Dorian was a phenomenal advisor to me, and I am extremely thankful to him for being patient, insightful, and, above all, supportive. I would also like to thank Asbj{\o}rn Nordentoft for sharing his expertise with me on numerous occasions and making many helpful comments and suggestions for this paper.\\

\section{Spectral decomposition}\label{fourierexpansionssection}
~\\
Define $A(z) = 2\pi i\int_{i\infty}^zf(w)\,dw$. The approach is to introduce the function $D(z,s,\chi)$, defined as
$$D(z,s,\chi) := \sum_{\Gamma_\infty\backslash\Gamma_0(N)}\chi(\gamma)A(\gamma z)\text{Im}(\gamma z)^s .$$
It's easily verified that $D(\alpha z,s,\chi) = \bar\chi(\alpha)D(z,s,\chi)$ and $E^*(z,s,\chi) = D(z,s,\chi) - A(z)E(z,s,\chi)$. Moreover $D(z,s,\chi)$ is $L^2$ for $\text{Re}(s) > 1$ (\cite{petris18} $\S$ 3).\\
\\
We give a Fourier expansion of $E^*(z,s,\chi)$ for even $\chi$ by getting an explicit spectral decomposition for $D(z,s,\chi)$ and using $E^*(z,s,\chi) = D(z,s,\chi) - A(z)E(z,s,\chi)$. The main obstacle to extending to non-prime $N$ and odd $\chi$ is obtaining Fourier expansions for $E(z,s,\chi)$ in those cases. Recent work of Young \cite{young17} may be helpful if one wishes to do this.\\
\\
Let $M_j(z)$, $j = 1,2,...$ be an orthogonal basis of Maass forms on $\Gamma_0(N)$ which transform as $M_j(\gamma z) = \bar\chi(\gamma)M_j(z)$ and normalized such that their Fourier coefficients are equal to their Hecke eigenvalues. The Selberg spectral decomposition \cite{hej83} for $D(z,s,\chi)$ then gives
$$D(z,s,\chi) = \sum_j \frac{\big\langle D(*,s,\chi), M_j\big\rangle}{\big\langle M_j,M_j\big\rangle} M_j(z) + \frac{1}{4\pi i}\sum_{\fa}\int_{\left(\frac{1}{2}\right)}\big\langle D(*,s,\chi), E_\fa(*,w,\chi)\big\rangle E_\fa(z,w,\chi) dw$$
where
$$\big\langle f,g\big\rangle = \int_{\Gamma_0(N)\backslash\mathfrak{h}}f(z)\overline{g(z)}\,\frac{dxdy}{y^2}$$
is the Petersson inner product, and
$$E_\fa(z,s,\chi) := \sum_{\Gamma_\fa\backslash\Gamma_0(N)}\chi(\gamma)\text{Im}(\sigma_\fa^{-1}\gamma z)^s ,$$
with $\sigma_\fa$ the matrix such that $\sigma_\fa\Gamma_\infty\sigma_\fa^{-1}$ is the stabilizer $\Gamma_\fa$ of the cusp $\fa$, and $\sigma_\fa\infty = \fa$. If $N$ is prime, there are only two cusps, $i\infty$ and $0$, and $\sigma_0 = \left(\begin{smallmatrix}0&-1\\N&0\end{smallmatrix}\right)$.\\
\\
When dealing with Fourier expansions of Eisenstein series it is often more convenient to work with ``completed'' Eisenstein series, given by $L(2s,\chi)E(z,s,\chi)$. Using completed Eisenstein series in place of regular Eisenstein series in the spectral decomposition formula instead gives
\begin{align*}
  D(z,s,\chi) = &\sum_j \frac{\big\langle D(*,s,\chi), M_j\big\rangle}{\big\langle M_j,M_j\big\rangle} M_j(z)\\
  + &\frac{1}{4\pi i}\int_{\left(\frac{1}{2}\right)}\frac{\big\langle D(*,s,\chi), E_{i\infty}(*,w,\chi)\big\rangle}{L(2w,\chi)} E_{i\infty}(z,w,\chi) dw\\
  + &\frac{1}{4\pi i}\int_{\left(\frac{1}{2}\right)}\frac{\big\langle D(*,s,\chi), E_0(*,w,\chi)\big\rangle}{L(2-2w,\bar\chi)} E_0(z,w,\chi) dw
\end{align*}
when $N$ is prime. We for the rest of this section we will use this form of the spectral decomposition, using completed Eisenstein series.\\
\\
\\
To proceed we need:
\begin{itemize}
\item A Fourier expansion for $M_j(z)$.
\item A Fourier expansion for $E_\fa(z,w,\chi)$.
\item The evaluation of $\big\langle D(*,s,\chi), M_j\big\rangle$.
\item The evaluation of $\big\langle D(*,s,\chi), E_\fa(*,w,\chi)\big\rangle$.
\end{itemize}
~\\
We give these in lemmas \ref{maassexpansion} through \ref{eisensteininnerprod0}. From this spectral expansion for $D(z,s,\chi)$ and the Fourier expansions of $E(z,s,\chi)$ and $M_j(z)$ we obtain the Fourier expansion of $D(z,s,\chi)$.\\
\\
From there, we can obtain the full Fourier expansion of $E^*(z,s,\chi)$. The $n^{\text{th}}$ Fourier coefficient is
$$\int_0^1 (D(z,s,\chi) - A(z)E(z,s,\chi))e^{-2\pi i n x} dx ,$$
which is easily evaluated using the Fourier expansions of $D(z,s,\chi)$, $E(z,s,\chi)$, and $f(z)$. This yields theorems \ref{constterm} and \ref{higherterms}.\\
\begin{lemma}\label{maassexpansion} The Fourier expansion of a Maass form $M_j$ of eigenvalue $1/4 + \lambda_j^2$ is given by
$$M_j(z) = \sum_{n\neq 0}c_{M_j}\!(n)y^{\frac{1}{2}}K_{i\lambda_j}(2\pi|n|y)e^{2\pi i n x}$$
where $K_v(y)$ is the $K$-Bessel function, defined as
$$K_v(y) = \frac{1}{2}\int_0^\infty\exp\!\left(\frac{1}{2}y(u + u^{-1})\right)u^v\frac{du}{u}.$$
\end{lemma}
~\\
\begin{proof} \cite{go06}, Theorem 3.5.1.\end{proof}
~\\
We normalize $M_j(z)$ so that $c_{M_j}(1) = 1$. The quantity $\langle M_j, M_j\rangle$ will appear later, and from the work of Hoffstein and Lockhart \cite{hoflock94} for all $\eps > 0$ we have the bound
$$N^{-\eps}\cosh(\pi \lambda_j)^{\frac{1}{2}} \ll \langle M_j, M_j\rangle^{-1} \ll N^{\eps}\cosh(\pi \lambda_j)^{\frac{1}{2}} .$$
\begin{lemma}\label{eisensteinexpansion} The Fourier expansions of the completed Eisenstein series $E_\fa(z,w,\chi)$ for even $\chi$ are given by
\begin{align*}
  &E_{i\infty}(z,w,\chi) = 2y^wL(2w,\chi) + \frac{4\tau(\chi)\pi^w}{N^{2w}\Gamma(w)}y^{\frac{1}{2}}\sum_{n\neq 0}|n|^{\frac{1}{2} - w}\sigma_{2w - 1}(n,\bar\chi)K_{w - \frac{1}{2}}(2\pi|n|y)e^{2\pi i n x}\\
  &\text{and}\\
  &E_0(z,w,\chi) = \frac{2\tau(\chi)\pi^{2w-1}\Gamma\left(1-w\right)}{N^{1-3w}\Gamma(w)}y^{1-w}L(2-2w,\bar\chi) + \frac{4\pi^w}{N^w\Gamma(w)}y^{\frac{1}{2}}\sum_{n\neq 0}|n|^{w-\frac{1}{2}}\sigma_{1-2w}(n,\chi)K_{w-\frac{1}{2}}(2\pi|n|y)e^{2\pi i n x} .
\end{align*}
\end{lemma}
\begin{proof}
\cite{go81} along with the identity
$$\Gamma(2w-1) = 4^{w-1}\pi^{-\frac{1}{2}}\Gamma\left(w-\frac{1}{2}\right)\Gamma(w)$$
and the functional equation
$$L(2w-1,\chi) = \frac{\tau(\chi)}{\sqrt{N}}\left(\frac{N}{\pi}\right)^{\frac{3}{2}-2w}\frac{\Gamma(1-w)}{\Gamma\left(w-\frac{1}{2}\right)}L(2-2w,\bar\chi) .$$
The reference gives a Fourier expansion for $E\left(\frac{-1}{Nz},w,\chi\right)$ instead of $E_0(z,w,\chi)$. However, for $\left(\begin{smallmatrix} a & b \\ c & d \end{smallmatrix}\right) \in \Gamma_0(N)$, we have
$$\begin{pmatrix} a & b \\ c & d\end{pmatrix} \begin{pmatrix} 0 & -1 \\ N & 0\end{pmatrix} = \begin{pmatrix} 0 & -1 \\ N & 0\end{pmatrix} \begin{pmatrix} d & -\frac{c}{N} \\ -bN & a\end{pmatrix} .$$
The rightmost matrix is an element of $\Gamma_0(N)$, and summing over all $\left(\begin{smallmatrix} a & b \\ c & d \end{smallmatrix}\right) \in \Gamma_0(N)$ is the same as summing over all $\left(\begin{smallmatrix} d & -\frac{c}{N} \\ -bN & a\end{smallmatrix}\right) \in \Gamma_0(N)$, so the Fourier expansion given is indeed for $E_0(z,s,\chi)$.
\end{proof}
\begin{lemma}\label{maassinnerprod} The inner product of $D(z,s,\chi)$ with the Maass form $M_j(z)$ is given by
$$\langle D(*,s,\chi), M_j\rangle = \frac{(4\pi)^{1-s}}{2}\frac{\Gamma\left(s+i\lambda_j-\frac{1}{2}\right)\Gamma\left(s-i\lambda_j-\frac{1}{2}\right)}{\Gamma(s)}L\left(s+\frac{1}{2},f\times\overline{M_j}\right) .$$
\end{lemma}
~\\
\begin{proof}
  We can evaluate $\langle D(*,s,\chi),M_j\rangle$ by unfolding $D(z,s,\chi)$:
\begin{align*}
  \langle D(*,s,\chi), M_j\rangle &= \int_{\Gamma_0(N)\backslash\mathfrak{h}}\sum_{\gamma\in\Gamma_\infty\backslash\Gamma_0(N)}\chi(\gamma)A(\gamma z)\text{Im}(\gamma z)^s\cdot \overline{M_j(z)}\frac{dxdy}{y^2}\\
  &= \int_0^1\int_0^\infty A(z)y^s\cdot \overline{M_j(z)}\frac{dxdy}{y^2}\\
  &= \int_0^1\int_0^\infty\sum_{n=1}^\infty \frac{a_n}{n}e^{-2\pi ny}e^{2\pi i n x}y^s\cdot \overline{c_{M_j}(n)}y^{\frac{1}{2}}\frac{1}{2}\int_0^\infty \exp\left(-\pi ny(u+u^{-1})\right)u^{-i\lambda_j-1}du\,e^{-2\pi i n x}\frac{dxdy}{y^2}\\
  &= \sum_{n=1}^\infty\frac{a_n \overline{c_{M_j}(n)}}{n}\frac{1}{2}\int_0^\infty\int_0^\infty\exp(-\pi ny(u + u^{-1} + 2)y^{s-\frac{3}{2}}u^{-i\lambda_j-1} du\,dy\\
  &= \sum_{n=1}^\infty\frac{a_n \overline{c_{M_j}(n)}}{n}\frac{1}{2}\int_0^\infty\frac{\Gamma\left(s-\frac{1}{2}\right)}{\left(\pi n(u+u^{-1}+2)\right)^{s-\frac{1}{2}}}u^{-i\lambda_j-1}du\\
  &= \frac{\pi^{\frac{1}{2}-s}}{2}\frac{\Gamma\left(s-\frac{1}{2}\right)\Gamma\left(s+i\lambda_j-\frac{1}{2}\right)\Gamma\left(s-i\lambda_j-\frac{1}{2}\right)}{\Gamma(2s-1)}\sum_{n=1}^\infty\frac{a_n \overline{c_{M_j}(n)}}{n^{s+\frac{1}{2}}}\\
  &= \frac{\pi^{\frac{1}{2}-s}}{2}\frac{\Gamma\left(s-\frac{1}{2}\right)\Gamma\left(s+i\lambda_j-\frac{1}{2}\right)\Gamma\left(s-i\lambda_j-\frac{1}{2}\right)}{\Gamma(2s-1)}L\left(s+\frac{1}{2},f\times\overline{M_j}\right)\\
  &= \frac{(4\pi)^{1-s}}{2}\frac{\Gamma\left(s+i\lambda_j-\frac{1}{2}\right)\Gamma\left(s-i\lambda_j-\frac{1}{2}\right)}{\Gamma(s)}L\left(s+\frac{1}{2},f\times\overline{M_j}\right) .
\end{align*}
The integral in $u$ was evaluated with Mathematica, and relies on the fact that the Mellin transform in $y$ of $K_v(y)$ is equal to $2^{\xi - 2}\Gamma\!\left(\frac{\xi + v}{2}\right)\Gamma\!\left(\frac{\xi - v}{2}\right)$, where $\xi$ is the variable of the Mellin transform \cite{grry}.
\end{proof}
\begin{lemma}\label{eisensteininnerprodinfty}The inner product of $D(z,s,\chi)$ with the Eisenstein series $E_{i\infty}(z,s,\chi)$ is given by
$$\langle D(*,s,\chi),E_{i\infty}(*,w,\chi)\rangle = 2\overline{\tau(\chi)}\pi^{\bar{w}-s+\frac{1}{2}}N^{-2\bar{w}}\frac{\Gamma\left(s-\frac{1}{2}\right)\Gamma(s-\bar{w})\Gamma(s+\bar{w}-1)}{\Gamma(2s-1)\Gamma(\bar{w})}\frac{L_f(s-\bar{w}+1,\chi)L_f(\bar{w}+s)}{L(2s,\chi)} .$$
\end{lemma}
\begin{proof}We can evaluate $\langle D(*,s,\chi),E_{i\infty}(*,w,\chi)\rangle$ by unfolding $D(z,s,\chi)$. We use $(...)^*$ to denote the complex conjugate of $(...)$.
  \begin{align*}
    &\langle D(z,s,\chi),E_{i\infty}(z,w,\chi)\rangle\\
    &= \int_{\Gamma_0(N)\backslash\mathfrak{h}}\sum_{\gamma\in\Gamma_\infty\backslash\Gamma_0(N)}\chi(\gamma)A(\gamma z)\text{Im}(\gamma z)\overline{E(\gamma z,w,\chi)}\frac{dz}{y^2}\\
    &=\int_0^\infty\int_0^1\sum_{n=1}^\infty \frac{a_n}{n}e^{2\pi inx}e^{-2\pi ny}y^{s}\left(2y^wL(2w,\chi) + \frac{4\pi^w\tau(\chi)}{N^{2w}\Gamma(w)}\sqrt{y}\sum_{n\neq0}|n|^{\frac{1}{2} - w}\sigma_{2w - 1}(n,\bar\chi)K_{w - \frac{1}{2}}(2\pi|n|y)e^{2\pi inx}\right)^*\frac{dxdy}{y^2}\\
  &=\int_0^\infty\sum_{n=1}^\infty \frac{a_n}{n}e^{-2\pi ny}y^{s-2}\frac{4\pi^{\bar{w}}\overline{\tau(\chi)}}{N^{2\bar{w}}\Gamma(\bar{w})}\sqrt{y}n^{\frac{1}{2} - \bar{w}}\sum_{d|n}\chi(d)d^{2\bar{w}-1}\frac{1}{2}\int_0^\infty \exp\left(-\pi ny(u+u^{-1})\right)u^{\bar{w}-\frac{3}{2}}du\,dy\\
  &=\frac{2\pi^{\bar{w}}\overline{\tau(\chi)}}{N^{2\bar{w}}\Gamma(\bar{w})}\sum_{n=1}^\infty a_n n^{-\frac{1}{2} - \bar{w}}\sum_{d|n}\chi(d)d^{2\bar{w}-1} \int_0^\infty\int_0^\infty \exp\left(-\pi ny(u+u^{-1}+2)\right)y^{s-\frac{3}{2}}u^{\bar{w}-\frac{3}{2}}du\,dy\\
&= \frac{2\pi^{\bar{w}}\overline{\tau(\chi)}}{N^{2\bar{w}}\Gamma(\bar{w})}\sum_{n=1}^\infty a_n n^{-\frac{1}{2} - \bar{w}}\sum_{d|n}\chi(d)d^{2\bar{w}-1} \int_0^\infty \frac{\Gamma\left(s-\frac{1}{2}\right)}{(n\pi(u+u^{-1}+2))^{s-\frac{1}{2}}}u^{\bar{w}-\frac{3}{2}}du\\
    &=\frac{2\pi^{\bar{w}-s+\frac{1}{2}}\overline{\tau(\chi)}\Gamma\left(s-\frac{1}{2}\right)}{N^{2\bar{w}}\Gamma(\bar{w})}\sum_{n=1}^\infty a_n n^{-s - \bar{w}}\sum_{d|n}\chi(d)d^{2\bar{w}-1} \int_0^\infty \frac{u^{\bar{w}-\frac{3}{2}}}{(u+u^{-1}+2)^{s-\frac{1}{2}}}du .
  \end{align*}
  The integral in $u$ can again be evaluated with Mathematica, which uses the Mellin transform of $K_v(y)$ mentioned in the proof of \ref{maassinnerprod}. The expression then becomes
  \begin{align*}
&= 2\pi^{\bar{w}-s+\frac{1}{2}}\overline{\tau(\chi)}N^{-2\bar{w}}\frac{\Gamma\left(s-\frac{1}{2}\right)\Gamma(s-\bar{w})\Gamma(s+\bar{w}-1)}{\Gamma(\bar{w})\Gamma(2s-1)}\sum_{n=1}^\infty a_n n^{-s - \bar{w}}\sum_{d|n}\chi(d)d^{2\bar{w}-1}\\
&= \sum_{d=1}^\infty\sum_{m=1}^\infty a(md)\chi(d)d^{2\bar{w}-1}(md)^{-s-\bar{w}} = \sum_{d=1}^\infty\sum_{m=1}^\infty a(md)\chi(d)d^{\bar{w}-s-1}m^{-s-\bar{w}}.
  \end{align*}
  \begin{align*}
\shortintertext{Consider the expression}
&\sum_{d=1}^\infty\sum_{m=1}^\infty \chi(d)d^{z_1}m^{z_2}a(m)a(d) .\\
\shortintertext{Using the Hecke relations this becomes}
&\sum_{d=1}^\infty\sum_{m=1}^\infty \chi(d)d^{z_1}m^{z_2}\sum_{r|(m,d)}ra\left(\frac{md}{r^2}\right) .\\
\shortintertext{Bringing the $r$ sum out turns this into}
&\sum_{r=1}^\infty\sum_{d=1}^\infty\sum_{m=1}^\infty \chi(r)r^{1+z_1+z_2}\chi(d)d^{z_1}m^{z_2}a(md) ,\\
\shortintertext{and so}
&\sum_{d=1}^\infty\sum_{m=1}^\infty \chi(d)d^{z_1}m^{z_2}a(md) = \frac{\sum_{d=1}^\infty\sum_{m=1}^\infty \chi(d)d^{z_1}m^{z_2}a(m)a(d)}{\sum_{r=1}^\infty\chi(r)r^{1+z_1+z_2}} .\\
\shortintertext{These are $L$-functions. Substituting $z_1 = \bar{w}-s-1, z_2 = -\bar{w}-s$ we get}
&\sum_{d=1}^\infty\sum_{m=1}^\infty a(md)\chi(d)d^{\bar{w}-s-1}m^{-s-\bar{w}} = \frac{L_f(s-\bar{w}+1,\chi)L_f(\bar{w}+s)}{L(2s,\chi)} .
  \end{align*}
Substituting this into our expression from earlier yields
$$\langle D(*,s,\chi),E_{i\infty}(*,w,\chi)\rangle = 2\overline{\tau(\chi)}\pi^{\bar{w}-s+\frac{1}{2}}N^{-2\bar{w}}\frac{\Gamma\left(s-\frac{1}{2}\right)\Gamma(s-\bar{w})\Gamma(s+\bar{w}-1)}{\Gamma(2s-1)\Gamma(\bar{w})}\frac{L_f(s-\bar{w}+1,\chi)L_f(\bar{w}+s)}{L(2s,\chi)} .$$
\end{proof}
\begin{lemma}\label{eisensteininnerprod0}The inner product of $D(z,s,\chi)$ with the Eisenstein series $E_0(z,s,\chi)$ is given by
$$\langle D(*,s,\chi),E_0(*,w,\chi)\rangle = 2\pi^{\bar{w}-s+\frac{1}{2}}N^{-\bar{w}}\frac{\Gamma\left(s-\frac{1}{2}\right)\Gamma(s-\bar{w})\Gamma(s+\bar{w}-1)}{\Gamma(2s-1)\Gamma(\bar{w})}\frac{L_f(s+\bar{w},\bar\chi)L_f(1-\bar{w}+s)}{L(2s,\bar\chi)} .$$
\end{lemma}
\begin{proof} The computation of $\langle D(*,s,\chi),E_0(*,w,\chi)\rangle$ is very similar to that of $\langle D(*,s,\chi),E_{i\infty}(*,w,\chi)$:
\begin{align*}\langle D(*,s,\chi),E_0(*,w,\chi)\rangle = \int_0^\infty\int_0^1&\sum_{n=1}^\infty \frac{a_n}{n}e^{2\pi inx}e^{-2\pi ny}y^{s}\left(\frac{2\tau(\chi)\pi^{2w-1}\Gamma\left(1-w\right)}{N^{1-3w}\Gamma(w)}y^{1-w}L(2-2w,\bar\chi)\right.\\
  &\left.+ \frac{4\pi^w}{N^w\Gamma(w)}\sqrt{y}\sum_{n\neq 0}|n|^{w-\frac{1}{2}}\sigma_{1-2w}(n,\chi)K_{w-\frac{1}{2}}(2\pi|n|y)e^{2\pi inx}\right)^*\frac{dxdy}{y^2} ,
\end{align*}
again using $(...)^*$ to denote the complex conjugate of $(...)$. Integrating in $x$:
\begin{align*}&= \int_0^\infty\sum_{n=1}^\infty \frac{a_n}{n}e^{-2\pi ny}y^{s-\frac{3}{2}}\frac{4\pi^{\bar{w}}}{N^{\bar{w}}\Gamma(\bar{w})}n^{\bar{w}-\frac{1}{2}}\sum_{d|n}\bar\chi(d)d^{1-2\bar{w}}\frac{1}{2}\int_0^\infty\exp\left(-\pi ny(u+u^{-1})\right)u^{\bar{w}-\frac{3}{2}}du\,dy\\
  &= \frac{2\pi^{\bar{w}}}{N^{\bar{w}}\Gamma(\bar{w})}\sum_{n=1}^\infty a_nn^{\bar{w}-\frac{3}{2}}\sum_{d|n}\bar\chi(d)d^{1-2\bar{w}}\int_0^\infty\int_0^\infty\exp\left(-\pi ny(u+u^{-1}+2)\right)y^{s-\frac{3}{2}}u^{\bar{w}-\frac{3}{2}}du\,dy .
\end{align*}
These integrals are identical to the ones that appear in the computation of $\langle D,E_{i\infty}\rangle$. The expression becomes
$$= 2\pi^{\bar{w}-s+\frac{1}{2}}N^{-\bar{w}}\frac{\Gamma\left(s-\frac{1}{2}\right)\Gamma(s-\bar{w})\Gamma(s+\bar{w}-1)}{\Gamma(\bar{w})\Gamma(2s-1)}\sum_{n=1}^\infty a_nn^{\bar{w}-s-1}\sum_{d|n}\bar\chi(d)d^{1-2\bar{w}} .$$
The sums are the same as the ones that appeared in the previous computation, but with $\bar{w}$ replaced with $1 - \bar{w}$ and $\chi$ replaced with $\bar\chi$. Hence
$$\langle D(*,s,\chi),E_0(*,w,\chi)\rangle = 2\pi^{\bar{w}-s+\frac{1}{2}}N^{-\bar{w}}\frac{\Gamma\left(s-\frac{1}{2}\right)\Gamma(s-\bar{w})\Gamma(s+\bar{w}-1)}{\Gamma(\bar{w})\Gamma(2s-1)}\frac{L_f(s+\bar{w},\bar\chi)L_f(1-\bar{w}+s)}{L(2s,\bar\chi)} .$$
\end{proof}
\noindent We can now prove theorem \ref{constterm}.
\begin{proof}[Proof of theorem \ref{constterm}] We compute directly:
\begin{align*}
  &\int_0^1E^*(x+iy,s,\chi)dx\\
  =&\int_0^1 D(x+iy,s,\chi) - A(x+iy)E(x+iy,s,\chi) \,dx\\
  =&\int_0^1 \left(\sum_j \frac{\langle D(*,s,\chi), M_j\rangle}{\langle M_j, M_j\rangle} M_j(x+iy) + \frac{1}{4\pi i}\sum_{\fa}\int_{\left(\frac{1}{2}\right)}\langle D(*,s,\chi), E_\fa(*,w,\chi)\rangle E_\fa(x+iy,w,\chi) dw\right)\,dx\\
  &\quad\quad\quad\quad - \int_0^1 A(x+iy)E(x+iy,s,\chi) \,dx .\\
\end{align*}
The Maass forms have no constant term, so they won't contribute to the expression at hand. Using lemmas \ref{eisensteinexpansion}, \ref{eisensteininnerprodinfty}, and \ref{eisensteininnerprod0} then gives
\begin{align*}
   &= \frac{1}{4\pi i}\int_{\left(\frac{1}{2}\right)}2\overline{\tau(\chi)}\pi^{\bar{w}-s+\frac{1}{2}}N^{-2\bar{w}}\frac{\Gamma\left(s-\frac{1}{2}\right)\Gamma(s-\bar{w})\Gamma(s+\bar{w}-1)}{\Gamma(2s-1)\Gamma(\bar{w})}\frac{L_f(s-\bar{w}+1,\chi)L_f(\bar{w}+s)}{L(2s,\chi)}\cdot 2y^w \,dw\\
  &+\frac{1}{4\pi i}\int_{\left(\frac{1}{2}\right)}2\pi^{\bar{w}-s+\frac{1}{2}}N^{-\bar{w}}\frac{\Gamma\left(s-\frac{1}{2}\right)\Gamma(s-\bar{w})\Gamma(s+\bar{w}-1)}{\Gamma(2s-1)\Gamma(\bar{w})}\frac{L_f(s+\bar{w},\bar\chi)L_f(1-\bar{w}+s)}{L(2s,\bar\chi)}\\
  &\quad\quad\quad\quad\quad\quad\quad\quad\quad\quad\quad\quad\quad\quad\quad\quad\quad\quad\quad\quad\quad\quad\quad\quad\quad\quad\cdot \frac{2\tau(\chi)\pi^{2w-1}\Gamma\left(1-w\right)}{N^{1-3w}\Gamma(w)}y^{1-w} \,dw .\\
  &-\sum_{m = 1}^{\infty}\frac{a_m}{m}e^{-2\pi m y}\frac{4\pi^s\tau(\chi)}{N^{2s}\Gamma(s)L(2s,\chi)}y^{\frac{1}{2}}m^{\frac{1}{2} - s}\sigma_{2s - 1}(m,\bar\chi)K_{s - \frac{1}{2}}(2\pi my)
\end{align*}
~\\
Note that $\bar w = 1 - w$ on the line $\text{Re}(w) = \frac{1}{2}$. Making this substitution and cleaning up a bit yields\\
\begin{align*}
  &\frac{1}{4\pi i}\int_{\left(\frac{1}{2}\right)}2\overline{\tau(\chi)}\pi^{1-w-s+\frac{1}{2}}N^{2w-2}\frac{\Gamma\left(s-\frac{1}{2}\right)\Gamma(s+w-1)\Gamma(s-w)}{\Gamma(2s-1)\Gamma(1-w)}\frac{L_f(s+w,\chi)L_f(s+1-w)}{L(2s,\chi)}\cdot 2y^w \,dw\\
  +&\frac{1}{4\pi i}\int_{\left(\frac{1}{2}\right)}2\tau(\chi)\pi^{w-s+\frac{1}{2}}N^{-2w}\frac{\Gamma\left(s-\frac{1}{2}\right)\Gamma(s-w)\Gamma(s+w-1)}{\Gamma(2s-1)\Gamma(w)}\frac{L_f(s+1-w,\bar\chi)L_f(s+w)}{L(2s,\bar\chi)}\cdot 2y^{1-w} \,dw\\
  -&\sum_{m = 1}^{\infty}\frac{a_m}{m}e^{-2\pi m y}\frac{4\pi^s\tau(\chi)}{N^{2s}\Gamma(s)L(2s,\chi)}y^{\frac{1}{2}}m^{\frac{1}{2} - s}\sigma_{2s - 1}(m,\bar\chi)K_{s - \frac{1}{2}}(2\pi my) .
\end{align*}
From \cite{osul00} we know that the constant term of $E^*(z,s,\chi)$ is of the form $\phi^*_\chi(s)y^{1-s}$. Therefore
$$\phi^*_\chi(s) = \lim_{y\to\infty}y^{s-1}\int_0^1 E^*(x+iy,s,\chi)\,dx .$$
As $y$ goes to infinity, the contribution from $A(z)E(z,s,\chi)$, the term in the last line, vanishes, because the $K$-Bessel functions have exponential decay in $y$.\\
\\
The contribution from $E_{i\infty}(z,s,\chi)$ on the first line can be evaluated by shifting the contour to the left. There are poles whenever $w = 1-s - n$ for $n$ a non-negative integer. After taking the above limit, the only term that will survive is when $w = 1-s$, which has residue
$$2\pi i \cdot -i\tau(\bar\chi)L_f(1,\chi) \pi^{-\frac{1}{2}}N^{-2s}\frac{\Gamma\left(s-\frac{1}{2}\right)}{\Gamma(s)}\frac{L_f(2s)}{L(2s,\chi)}y^{1-s} .$$
Similarly, the contribution from $E_0(z,s,\chi)$ can be evaluated by shifting the integral in the second line to the right. The only residue which will contribute after the limit in $y$ is when $w = s$, and is equal to
$$-2\pi i \cdot -i\tau(\chi)L_f(1,\bar\chi) \pi^{-\frac{1}{2}}N^{-2s}\frac{\Gamma\left(s-\frac{1}{2}\right)}{\Gamma(s)}\frac{L_f(2s)}{L(2s,\bar\chi)}y^{1-s} .$$
From this we can conclude that for any $y$, we have
\begin{align*}
  \int_0^1 E^*(x + iy,s,\chi)\, dx = &\left(2\tau(\bar\chi)L_f(1,\chi) \pi^{\frac{1}{2}}N^{-2s}\frac{\Gamma\left(s-\frac{1}{2}\right)}{\Gamma(s)}\frac{L_f(2s)}{L(2s,\chi)}\right.\\
  &\left.- 2\tau(\chi)L_f(1,\bar\chi) \pi^{\frac{1}{2}}N^{-2s}\frac{\Gamma\left(s-\frac{1}{2}\right)}{\Gamma(s)}\frac{L_f(2s)}{L(2s,\bar\chi)}\right)y^{1-s} .
\end{align*}
\end{proof}
\noindent
We can also prove theorem \ref{higherterms} in a similar way.
\begin{proof}[Proof of theorem \ref{higherterms}] We compute directly:
\begin{align*}
  &\int_0^1E^*(x+iy,s,\chi)e^{-2\pi i n x}dx\\
  =&\int_0^1 (D(x+iy,s,\chi) - A(x+iy)E(x+iy,s,\chi))e^{-2\pi i n x} dx\\
  =&\int_0^1 \left(\sum_j \frac{\langle D(*,s,\chi), M_j\rangle}{\langle M_j, M_j\rangle} M_j(x+iy) + \frac{1}{4\pi i}\sum_{\fa}\int_{\left(\frac{1}{2}\right)}\langle D(*,s,\chi), E_\fa(*,w,\chi)\rangle E_\fa(x+iy,w,\chi) dw\right)e^{-2\pi i n x}dx\\
  &\quad\quad\quad\quad - \int_0^1 A(x+iy)E(x+iy,s,\chi)e^{-2\pi i n x} dx\\
  =&\sum_j \frac{(4\pi)^{1-s}}{2\langle M_j, M_j\rangle}\frac{\Gamma\left(s+i\lambda_j-\frac{1}{2}\right)\Gamma\left(s-i\lambda_j-\frac{1}{2}\right)}{\Gamma(s)}L\left(s+\frac{1}{2},f\times\overline{M_j}\right) \cdot c_{M_j}(n)y^{\frac{1}{2}}K_{i\lambda_j}(2\pi|n|y)\\
  +&\frac{1}{4\pi i}\int_{\left(\frac{1}{2}\right)}4\overline{\tau(\chi)}\pi^{\bar{w}-s+\frac{1}{2}}N^{-2\bar{w}}\frac{\Gamma\left(s-\frac{1}{2}\right)\Gamma(s-\bar{w})\Gamma(s+\bar{w}-1)}{\Gamma(2s-1)\Gamma(\bar{w})}\frac{L_f(s-\bar{w}+1,\chi)L_f(\bar{w}+s)}{L(2s,\chi)}\\
  &\quad\quad\quad\quad\quad\quad\quad\quad\quad\quad\quad\quad\quad\quad\quad\quad\quad\quad\quad\quad\quad\quad\cdot \frac{4\pi^w\tau(\chi)}{N^{2w}\Gamma(w)L(2w,\chi)}y^{\frac{1}{2}}|n|^{\frac{1}{2} - w}\sigma_{2w - 1}(n,\bar\chi)K_{w - \frac{1}{2}}(2\pi|n|y) \,dw\\
  +&\frac{1}{4\pi i}\int_{\left(\frac{1}{2}\right)}4\pi^{\bar{w}-s+\frac{1}{2}}N^{-\bar{w}}\frac{\Gamma\left(s-\frac{1}{2}\right)\Gamma(s-\bar{w})\Gamma(s+\bar{w}-1)}{\Gamma(2s-1)\Gamma(\bar{w})}\frac{L_f(s+\bar{w},\bar\chi)L_f(1-\bar{w}+s)}{L(2s,\bar\chi)}\\
  &\quad\quad\quad\quad\quad\quad\quad\quad\quad\quad\quad\quad\quad\quad\quad\quad\quad\quad\quad\quad\quad\quad\cdot \frac{4\pi^w}{N^w\Gamma(w)L(2-2w,\bar\chi)}y^{\frac{1}{2}}|n|^{w-\frac{1}{2}}\sigma_{1-2w}(n,\chi)K_{w-\frac{1}{2}}(2\pi|n|y) \,dw\\
  -&\sum_{m = 1}^{\infty}\frac{a_m}{m}e^{-2\pi m y}\frac{4\pi^s\tau(\chi)}{N^{2s}\Gamma(s)L(2s,\chi)}y^{\frac{1}{2}}|n-m|^{\frac{1}{2} - s}\sigma_{2s - 1}(n-m,\bar\chi)K_{s - \frac{1}{2}}(2\pi|n-m|y) .
\end{align*}
~\\
Again note that $\bar w = 1 - w$ on the line $\text{Re}(w) = \frac{1}{2}$. Making this substitution and cleaning up a bit yields
\begin{align*}
  &\sum_j \frac{(4\pi)^{1-s}}{2\langle M_j, M_j\rangle}\frac{\Gamma\left(s+i\lambda_j-\frac{1}{2}\right)\Gamma\left(s-i\lambda_j-\frac{1}{2}\right)}{\Gamma(s)}L\left(s+\frac{1}{2},f\times\overline{M_j}\right) \cdot c_{M_j}(n)y^{\frac{1}{2}}K_{i\lambda_j}(2\pi|n|y)\\
  +&\frac{1}{4\pi i}\int_{\left(\frac{1}{2}\right)}16\pi^{\frac{3}{2}-s}N^{-1}\frac{\Gamma\left(s-\frac{1}{2}\right)\Gamma(s+w-1)\Gamma(s-w)}{\Gamma(2s-1)\Gamma(1-w)\Gamma(w)}\frac{L_f(s+w,\chi)L_f(s+1-w)}{L(2s,\chi)L(2w,\chi)}\\
  &\quad\quad\quad\quad\quad\quad\quad\quad\quad\quad\quad\quad\quad\quad\quad\quad\quad\quad\quad\quad\quad\quad\quad\quad\quad\quad\quad\quad\cdot y^{\frac{1}{2}}|n|^{\frac{1}{2} - w}\sigma_{2w - 1}(n,\bar\chi)K_{w - \frac{1}{2}}(2\pi|n|y) \,dw\\
  +&\frac{1}{4\pi i}\int_{\left(\frac{1}{2}\right)}16\pi^{\frac{3}{2}-s}N^{-1}\frac{\Gamma\left(s-\frac{1}{2}\right)\Gamma(s-w)\Gamma(s+w-1)}{\Gamma(2s-1)\Gamma(w)\Gamma(1-w)}\frac{L_f(s+1-w,\bar\chi)L_f(s+w)}{L(2s,\bar\chi)L(2-2w,\bar\chi)}\\
  &\quad\quad\quad\quad\quad\quad\quad\quad\quad\quad\quad\quad\quad\quad\quad\quad\quad\quad\quad\quad\quad\quad\quad\quad\quad\quad\quad\quad\cdot y^{\frac{1}{2}}|n|^{w-\frac{1}{2}}\sigma_{1-2w}(n,\chi)K_{w-\frac{1}{2}}(2\pi|n|y) \,dw\\
  -&\sum_{m = 1}^{\infty}\frac{a_m}{m}e^{-2\pi m y}\frac{4\pi^s\tau(\chi)}{N^{2s}\Gamma(s)L(2s,\chi)}y^{\frac{1}{2}}|n-m|^{\frac{1}{2} - s}\sigma_{2s - 1}(n-m,\bar\chi)K_{s - \frac{1}{2}}(2\pi|n-m|y) .
\end{align*}
Making the substitution $w \mapsto 1-w$ in the third line and using the identity
$$\Gamma(2s-1) = 4^{w-1}\pi^{-\frac{1}{2}}\Gamma\left(s-\frac{1}{2}\right)\Gamma(s)$$
then gives\\
\begin{align*}
  &\int_0^1E^*(x+iy,s,\chi)e^{-2\pi i n x}dx\\
  =&\sum_j \frac{(4\pi)^{1-s}}{2\langle M_j, M_j\rangle}\frac{\Gamma\left(s+i\lambda_j-\frac{1}{2}\right)\Gamma\left(s-i\lambda_j-\frac{1}{2}\right)}{\Gamma(s)}L\left(s+\frac{1}{2},f\times\overline{M_j}\right) \cdot c_{M_j}(n)y^{\frac{1}{2}}K_{i\lambda_j}(2\pi|n|y)\\
  -&\sum_{m = 1}^{\infty}\frac{a_m}{m}e^{-2\pi m y}\frac{4\pi^s\tau(\chi)}{N^{2s}\Gamma(s)L(2s,\chi)}|n-m|^{\frac{1}{2} - s}\sigma_{2s - 1}(n-m,\bar\chi)y^{\frac{1}{2}}K_{s - \frac{1}{2}}(2\pi|n-m|y)\\
  +&\frac{2^{2s+1}\pi^{2-s}}{2\pi i \cdot N\Gamma(s)}\int_{\left(\frac{1}{2}\right)}\frac{\Gamma(s+w-1)\Gamma(s-w)}{\Gamma(1-w)\Gamma(w)}L_f(s+1-w)\left(\frac{L_f(s+w,\chi)}{L(2s,\chi)L(2w,\chi)}\sigma_{2w - 1}(n,\bar\chi) - \frac{L_f(s+w,\bar\chi)}{L(2s,\bar\chi)L(2w,\bar\chi)}\sigma_{2w - 1}(n,\chi)\right)\\
  &\quad\quad\quad\quad\quad\quad\quad\quad\quad\quad\quad\quad\quad\quad\quad\quad\quad\quad\quad\quad\quad\quad\quad\quad\quad\quad\quad\quad\quad\quad\quad\quad\quad\cdot |n|^{\frac{1}{2} - w}y^{\frac{1}{2}}K_{w - \frac{1}{2}}(2\pi|n|y) \,dw .
\end{align*}
Shifting the contour to the left allows us to write the integral as a sum of residues. The expression becomes 
  \begin{align*}
  &\int_0^1E^*(x+iy,s,\chi)e^{-2\pi i n x}dx\\
  =&\sum_j \frac{(4\pi)^{1-s}}{2\langle M_j, M_j\rangle}\frac{\Gamma\left(s+i\lambda_j-\frac{1}{2}\right)\Gamma\left(s-i\lambda_j-\frac{1}{2}\right)}{\Gamma(s)}L\left(s+\frac{1}{2},f\times\overline{M_j}\right) \cdot c_{M_j}(n)y^{\frac{1}{2}}K_{i\lambda_j}(2\pi|n|y)\\
  -&\sum_{m = 1}^{\infty}\frac{a_m}{m}e^{-2\pi m y}\frac{4\pi^s\tau(\chi)}{N^{2s}\Gamma(s)}|n-m|^{\frac{1}{2} - s}\sigma_{2s - 1}(n-m,\bar\chi)y^{\frac{1}{2}}K_{s - \frac{1}{2}}(2\pi|n-m|y)\\
  +&\frac{2^{2s+1}\pi^{2-s}}{N\Gamma(s)}\sum_{k=0}^\infty\frac{\Gamma(2s+k-1)}{k!\Gamma(s+k)\Gamma(1-s-k)}\\
  &\quad\quad\quad\quad\quad\quad\quad\cdot L_f(k)\left(\frac{L_f(1-k,\chi)}{L(2s,\chi)L(2-2s-2k,\chi)}\sigma_{1 - 2s - 2k}(n,\bar\chi) - \frac{L_f(1-k,\bar\chi)}{L(2s,\bar\chi)L(2-2s-2k,\bar\chi)}\sigma_{1 - 2s - 2k}(n,\chi)\right)\\
  &\quad\quad\quad\quad\quad\quad\quad\quad\quad\quad\quad\quad\quad\quad\cdot |n|^{s+k-\frac{1}{2}}y^{\frac{1}{2}}K_{\frac{1}{2}-s-k}(2\pi|n|y)\\
  +&\frac{2^{2s+1}\pi^{2-s}}{N\Gamma(s)}\!\!\sum_{\rho: L(2\rho,\chi) = 0}\! \frac{\Gamma(s+\rho-1)\Gamma(s-\rho)}{\Gamma(1-\rho)\Gamma(\rho)}\frac{L_f(s+1-\rho)L_f(s+\rho,\chi)}{L(2s,\chi)}\underset{\! w = \rho}{\text{{Res }}}\frac{1}{L(2w,\chi)}\,\sigma_{2\rho - 1}(n,\bar\chi)|n|^{\frac{1}{2} - \rho}y^{\frac{1}{2}}K_{\rho - \frac{1}{2}}(2\pi|n|y)\\
  -&\frac{2^{2s+1}\pi^{2-s}}{N\Gamma(s)}\!\!\sum_{\rho: L(2\rho,\bar\chi) = 0}\! \frac{\Gamma(s+\rho-1)\Gamma(s-\rho)}{\Gamma(1-\rho)\Gamma(\rho)}\frac{L_f(s+1-\rho)L_f(s+\rho,\bar\chi)}{L(2s,\bar\chi)}\underset{\! w = \rho}{\text{{Res }}}\frac{1}{L(2w,\bar\chi)}\,\sigma_{2\rho - 1}(n,\chi)|n|^{\frac{1}{2} - \rho}y^{\frac{1}{2}}K_{\rho - \frac{1}{2}}(2\pi|n|y).\\
\end{align*}
In the proof of theorem \ref{constterm} we used an argument to be able to discard all but one of the terms appearing, but in the current situation no argument of that sort is immediately available to us.
\end{proof}

\section{Sums of modular symbols twisted by Dirichlet characters}\label{applicationssection}
\subsection{Geometric ordering}
It was conjectured in \cite{go99} and proved in \cite{goosul03} that
$$\sum_{\gamma: |\!|\gamma|\!|_i \leq X} \langle \gamma,f\rangle \sim \underset{\! w=1}{\text{Res }}E^*(z,w)\cdot X$$
where $|\!|\gamma|\!|_z := |cz+d|^2$. Using the Fourier expansion of $E^*(z,s,\chi)$ we can obtain similar results when the sum on the left is twisted by $\chi$:
$$\sum_{\gamma: |\!|\gamma|\!|_z \leq \text{Im}(z)X} \chi(\gamma)\langle \gamma,f\rangle = \sum_{s:\text{ poles of }E^*(z,s,\chi)} \underset{\! w=s}{\text{Res }}E^*(z,w,\chi)\cdot X^s .$$
This sum can be evaluated using theorems \ref{constterm} and \ref{higherterms}. Via inspection of these theorems, we see that the main term will be $\cO(X^{\frac{1}{2} + \frac{1}{2}\text{Re}(\rho)})$ where $\rho$ is the rightmost zero of $L(s,\chi)$, and the most significant error terms will be $\cO(X^{\frac{1}{2}})$ and come from the Maass part of the spectral decomposition of $E^*(z,s,\chi)$. The Riemann hypothesis for Dirichlet $L$-functions implies that the main term will be of size $\cO(X^{\frac{3}{4}})$. 
\begin{proof}[Proof of theorem \ref{geometricordering}]
  For $\sigma > 2$ we evaluate $\int_{(\sigma)}E^*(z,s,\chi)X^s\frac{ds}{s}$ in two different ways. First, using the series expansion
  \begin{align*}
E^*(z,s,\chi) \,=& \sum_{\gamma \in \Gamma_\infty\backslash\Gamma_0(N)}\chi(\gamma)\langle\gamma,f\rangle\frac{\text{Im}(z)^s}{|cz+d|^{2s}}\\
\shortintertext{we have}
  \int_{(\sigma)}E^*(z,s,\chi)X^s\frac{ds}{s} \,=&\sum_{\gamma \in \Gamma_\infty\backslash\Gamma_0(N)}\chi(\gamma)\langle\gamma,f\rangle\int_{(\sigma)}\frac{\text{Im}(z)^s}{|cz+d|^{2s}}X^s\frac{ds}{s}\\
  \,=&\sum_{\gamma:|\!|\gamma|\!|_z \leq \text{Im}(z)X}\chi(\gamma)\langle\gamma,f\rangle .
\end{align*}
On the other hand, we can also evaluate this integral as a sum of residues at the poles of $E^*(z,s,\chi)$. Because $K_v(2\pi|n|y)$ has exponential decay in $2\pi|n|y$, the only poles of $E^*(z,s,\chi)$ will be the poles of its individual Fourier coefficients. Computing the residues at the most significant poles, we obtain
  \begin{align*}
  &\,\,\sum_{\substack{\gamma \in \Gamma_\infty\backslash\Gamma_0(N)\\|\!|\gamma|\!|_z \leq \text{{Im}}(z)X}}\chi(\gamma)\langle\gamma,f\rangle\\ 
  &=\sum_{\rho: L(\rho,\bar\chi) = 0} \frac{(4\pi)^{1 - \frac{\rho}{2}}}{N}\frac{\Gamma(1 - \rho)}{\Gamma\!\left(1 - \frac{\rho}{2}\right)^2\Gamma\!\left(\frac{\rho}{2}\right)}\frac{L_f(0)L_f(1,\bar\chi)}{L(1 - \rho,\bar\chi)}\underset{\! s = \rho}{\text{{Res }}}\frac{1}{L(s,\bar\chi)}\sum_{n\neq 0}e^{2\pi i n x}\sigma_{\rho}(n,\chi)|n|^{\frac{1-\rho}{2}}y^{\frac{1}{2}}K_{\frac{1}{2}+\rho}(2\pi|n|y) \cdot X^{1 - \frac{\rho}{2}}\\
  &- \sum_{\rho: L(\rho,\chi) = 0} \frac{(4\pi)^{1 - \frac{\rho}{2}}}{N}\frac{\Gamma(1 - \rho)}{\Gamma\!\left(1 - \frac{\rho}{2}\right)^2\Gamma\!\left(\frac{\rho}{2}\right)}\frac{L_f(0)L_f(1,\chi)}{L(1 - \rho,\chi)}\underset{\! s = \rho}{\text{{Res }}}\frac{1}{L(s,\chi)}\sum_{n\neq 0}e^{2\pi i n x}\sigma_{\rho}(n,\bar\chi)|n|^{\frac{1-\rho}{2}}y^{\frac{1}{2}}K_{\frac{1}{2}+\rho}(2\pi|n|y) \cdot X^{1 - \frac{\rho}{2}}\\
  &+ \,\,\,\,\cO(X^{\frac{1}{2}}).
\end{align*}
Note that there are no poles coming from the trivial zeroes of $L(s,\chi)$ because of the gamma factors.\end{proof}
\subsection{Arithmetic ordering}
The Eisenstein series $E^*(z,s,\chi)$ admits a double coset decomposition involving the function $\phi^*_\chi(n,s)$ defined in the introduction. In the classical case where the modular symbols are absent, the Kloosterman sums analogous to $\phi^*_\chi(n,s)$ can expressed in a closed form, so it is natural to ask if the same can be done for $\phi^*_\chi(n,s)$. This is done in corollary \ref{kloosterman}, which we now prove.\\
\begin{proof}[Proof of corollary \ref{kloosterman}] From the work of O'Sullivan \cite{osul00}, we have
  $$E^*(z,s,\chi) = \phi^*_\chi(0,s)y^{1-s} + \sum_{n\neq 0}\phi^*_\chi(n,s)\cdot 2|n|^{\frac{1}{2}}y^{\frac{1}{2}}K_{s-\frac{1}{2}}(2\pi|n|y)e^{2\pi i n x}.$$
  Hence, for $n \neq 0$, we have
  $$\int_0^1 E^*(z,s,\chi)e^{-2\pi i n x}\,dx = \phi^*_\chi(n,s)\cdot 2|n|^{\frac{1}{2}}y^{\frac{1}{2}}K_{s-\frac{1}{2}}(2\pi|n|y).$$
  On the other hand, from theorem \ref{higherterms} we have
\begin{align*}
  &\int_0^1E^*(x+iy,s,\chi)e^{-2\pi i n x}dx\\
  =&\sum_j \frac{(4\pi)^{1-s}}{2\langle M_j, M_j\rangle}\frac{\Gamma\left(s+i\lambda_j-\frac{1}{2}\right)\Gamma\left(s-i\lambda_j-\frac{1}{2}\right)}{\Gamma(s)}L\left(s+\frac{1}{2},f\times\overline{M_j}\right) \cdot c_{M_j}(n)y^{\frac{1}{2}}K_{i\lambda_j}(2\pi|n|y)\\
  -&\sum_{m = 1}^{\infty}\frac{a_m}{m}e^{-2\pi m y}\frac{4\pi^s\tau(\chi)}{N^{2s}\Gamma(s)}|n-m|^{\frac{1}{2} - s}\sigma_{2s - 1}(n-m,\bar\chi)y^{\frac{1}{2}}K_{s - \frac{1}{2}}(2\pi|n-m|y)\\
  +&\frac{2^{2s+1}\pi^{2-s}}{N\Gamma(s)}\sum_{k=0}^\infty\frac{\Gamma(2s+k-1)}{k!\Gamma(s+k)\Gamma(1-s-k)}\\
  &\quad\quad\quad\quad\quad\quad\quad\cdot L_f(k)\left(\frac{L_f(1-k,\chi)}{L(2s,\chi)L(2-2s-2k,\chi)}\sigma_{1 - 2s - 2k}(n,\bar\chi) - \frac{L_f(1-k,\bar\chi)}{L(2s,\bar\chi)L(2-2s-2k,\bar\chi)}\sigma_{1 - 2s - 2k}(n,\chi)\right)\\
  &\quad\quad\quad\quad\quad\quad\quad\quad\quad\quad\quad\quad\quad\quad\cdot |n|^{s+k-\frac{1}{2}}y^{\frac{1}{2}}K_{\frac{1}{2}-s-k}(2\pi|n|y)\\
  +&\frac{2^{2s+1}\pi^{2-s}}{N\Gamma(s)}\!\!\sum_{\rho: L(2\rho,\chi) = 0}\! \frac{\Gamma(s+\rho-1)\Gamma(s-\rho)}{\Gamma(1-\rho)\Gamma(\rho)}\frac{L_f(s+1-\rho)L_f(s+\rho,\chi)}{L(2s,\chi)}\underset{\! w = \rho}{\text{{Res }}}\frac{1}{L(2w,\chi)}\,\sigma_{2\rho - 1}(n,\bar\chi)|n|^{\frac{1}{2} - \rho}y^{\frac{1}{2}}K_{\rho - \frac{1}{2}}(2\pi|n|y)\\
  -&\frac{2^{2s+1}\pi^{2-s}}{N\Gamma(s)}\!\!\sum_{\rho: L(2\rho,\bar\chi) = 0}\! \frac{\Gamma(s+\rho-1)\Gamma(s-\rho)}{\Gamma(1-\rho)\Gamma(\rho)}\frac{L_f(s+1-\rho)L_f(s+\rho,\bar\chi)}{L(2s,\bar\chi)}\underset{\! w = \rho}{\text{{Res }}}\frac{1}{L(2w,\bar\chi)}\,\sigma_{2\rho - 1}(n,\chi)|n|^{\frac{1}{2} - \rho}y^{\frac{1}{2}}K_{\rho - \frac{1}{2}}(2\pi|n|y).
\end{align*}
  This immediately implies the desired result for $n \neq 0$. The case where $n = 0$ is proved similarly.
\end{proof}
~\\
We can now prove theorem \ref{arithmeticordering}.
\begin{proof}[Proof of theorem \ref{arithmeticordering}]
  For $\sigma > 2$ we have
  \begin{align*}
    \int_{(\sigma)}\sum_{\substack{\gamma \in \Gamma_\infty\backslash\Gamma_0(N)/\Gamma_\infty}}\frac{\chi(\gamma)\langle\gamma,f\rangle}{|c|^{2s}} X^s\frac{ds}{s}
    = &\sum_{\substack{\gamma \in \Gamma_\infty\backslash\Gamma_0(N)/\Gamma_\infty}}\chi(\gamma)\langle\gamma,f\rangle \int_{(\sigma)}\left(\frac{X}{c^2}\right)^s\frac{ds}{s}\\
    &\sum_{\substack{\gamma \in \Gamma_\infty\backslash\Gamma_0(N)/\Gamma_\infty\\|c|\, <\, X^{\frac{1}{2}}}}\chi(\gamma)\langle\gamma,f\rangle
  \end{align*}
  On the other hand, using corollary \ref{kloosterman}, we have
  \begin{align*}
    \int_{(\sigma)}\sum_{\substack{\gamma \in \Gamma_\infty\backslash\Gamma_0(N)/\Gamma_\infty}}\frac{\chi(\gamma)\langle\gamma,f\rangle}{|c|^{2s}} X^s\frac{ds}{s}
    = \int_{(\sigma)} 2N^{-2s}L_f(2s)\left(\frac{\tau(\bar\chi)L_f(1,\chi)}{L(2s,\chi)}-\frac{\tau(\chi)L_f(1,\bar\chi)}{L(2s,\bar\chi)}\right)X^s\frac{ds}{s}.
  \end{align*}
  Shifting the contour to the left then gives the desired result.
\end{proof}
  
\bibliographystyle{plain}
\bibliography{modsymbbib}{}

\end{document}